\documentclass[11pt, reqno]{amsart}
\usepackage{lineno,hyperref}

\usepackage{cases}
\usepackage[square, comma, sort&compress, numbers]{natbib}
\usepackage{float}
\usepackage{pifont}
\usepackage{bbding}
\usepackage{amssymb}
\usepackage{mathrsfs}
\usepackage{subfigure}
\usepackage{caption}
\usepackage{geometry}
\usepackage{color}
\usepackage{amsmath, amsthm, amscd, amsfonts, amssymb, graphicx, color}
\usepackage{lineno}
\newtheorem{theorem}{Theorem}

\newtheorem{lemma}{Lemma}

\newtheorem{cor}{Corollary}
\newtheorem{con}{Conjecture}
\newtheorem{dfn}{Definition}
\newtheorem{rem}{Remark}

\numberwithin{equation}{section}

\newcommand{\abs}[1]{\left\vert#1\right\vert}

\newcommand{\N}{\mbox{$\mathbb{N}$}}

\newcommand{\C}{\mbox{$\mathbb{C}$}}

\newcommand{\D}{\mbox{$\mathbb{D}$}}

\date{\today}
\begin{document}
\setcounter{page}{1}

\title[Quasiconformal close-to-convex harmonic mappings]
{On quasiconformal close-to-convex harmonic\\ \vskip.05in mappings involving starlike functions}

\author[Zhi-Gang Wang, Xin-Zhong Huang, Zhi-Hong Liu and Rahim Kargar]{Zhi-Gang Wang, Xin-Zhong Huang, Zhi-Hong Liu and Rahim Kargar}

\vskip.10in
\address{\noindent Zhi-Gang Wang\vskip.05in
 School of Mathematics and Computing Science, Hunan
First Normal University, Changsha 410205, Hunan, P. R. China.}
\vskip.05in
\email{\textcolor[rgb]{0.00,0.00,0.84}{wangmath$@$163.com}}

\address{\noindent Xin-Zhong Huang\vskip.05in
School of Mathematical Sciences, Huaqiao University, Quanzhou 362021, Fujian,
P. R. China.}
\email{\textcolor[rgb]{0.00,0.00,0.84}{huangxz$@$hqu.edu.cn}}

\address{\noindent Zhi-Hong Liu\vskip.05in
College of Science, Guilin University of Technology, Guilin 541004, Guangxi, P. R. China.}
\email{\textcolor[rgb]{0.00,0.00,0.84}{liuzhihongmath$@$163.com}}

\address{\noindent Rahim Kargar\vskip.05in
Department of Mathematics and Statistics, University of Turku, Turku, Finland.}
\email{\textcolor[rgb]{0.00,0.00,0.84}{rakarg$@$utu.fi}}




\begin{abstract}
In the present paper, we discuss several basic properties of a class of quasiconformal close-to-convex harmonic mappings with starlike analytic part, such results as coefficient inequalities, an integral representation, a growth theorem, an area theorem, and radii of close-to-convexity of partial sums of the class, are derived.
\end{abstract}

\vskip.20in

\maketitle

\tableofcontents
   
\section{Introduction}

A planar harmonic mapping $f$ in the open unit disk $\D$ can be represented as $f=h+\overline{g}$, where $h$ and $g$ are analytic
functions in $\D$. We call $h$ and $g$ the analytic part and
co-analytic part of $f$, respectively. Since the Jacobian of $f$ is
given by $\abs{h'}^2-\abs{g'}^2$, by Lewy's theorem (see
\cite{Lewy}), it is locally univalent and sense-preserving if and
only if $|g^{\prime}|<|h^{\prime}|$, or equivalently, if
$h^{\prime}(z)\neq 0$ and the dilatation
$\omega={g^{\prime}}/{h^{\prime}}$ has the property $|\omega|<1$ in
$\D$. Let $\mathcal{H}$ denote the class of harmonic functions
$f=h+\overline{g}$ normalized by the conditions $f(0)=f_{z}(0)-1=0$, which have the
form
\begin{equation}\label{111}f (z)=z+\sum_{k=2}^{\infty}a_kz^k+\overline{\sum_{k=1}^{\infty}b_kz^k} \quad (z\in\D).\end{equation}
Denote by $\mathcal{S_{\mathcal{H}}}$ the class of harmonic functions $f\in\mathcal{H}$
that are univalent and sense-preserving in $\D$. Also denote by $\mathcal{S}^{0}_{\mathcal{H}}$ the subclass of $\mathcal{S_{\mathcal{H}}}$ with the additional condition $f_{\overline{z}}(0)=0$. We observe that Clunie and Sheil-Small \cite{cs} have proved several fundamental characteristics for the class $\mathcal{S_{\mathcal{H}}}$, but other basic problems such as Riemann mapping theorem for planar harmonic mappings, harmonic analogue of Bieberbach conjecture, sharp coefficient inequalities and radius of covering theorem for the class $\mathcal{S}^{0}_{\mathcal{H}}$ are still \textit{open} (see \cite{d}). The classical family $\mathcal{S}$ of analytic univalent and
normalized functions in $\D$ is a subclass of
$\mathcal{S}_\mathcal{H}^0$ with $g(z)\equiv 0$.

If a univalent harmonic mapping $f=h+\overline{g}$ satisfies the condition
\begin{equation*}\label{13}
\abs{\omega(z)}=\abs{\frac{g^{\prime}(z)}{h^{\prime}(z)}}\leq k\quad (0\leq k<1;\, z\in\D),
\end{equation*}
then $f$ is said to be a $K$-quasiconformal harmonic mapping, where $$K=\frac{1+k}{1-k}\quad (0\leq k<1).$$

A domain $\Omega$ is said to be close-to-convex if
$\C\backslash\Omega$ can be represented as a union of non-crossing
half-lines. Following the result due to Kaplan (see \cite{k}), an
analytic function $f$ is called close-to-convex if there exits a
univalent convex function $\phi$ defined in $\D$ such that
$${\rm Re}\left(\frac{f'(z)}{\phi'(z)}\right)>0\quad (z\in\D).$$ Furthermore, a planar harmonic
mapping $f:\D\rightarrow\C$ is close-to-convex if it is injective
and $f(\D)$ is a close-to-convex domain. We denote by $\mathcal{C}_\mathcal{H}^0$ the class of close-to-convex harmonic mappings.

The theory and applications of planar harmonic mappings are
presented in the recent monograph by Duren \cite{d}. Furthermore,
Bshouty
{\it et al.} \cite{bjj,bl,bls}, Chen \textit{et al.} \cite{cprw}, Chuaqui and
Hern\'{a}ndez \cite{ch}, Kalaj \cite{k1}, Mocanu \cite{m4,m5}, Nagpal and Ravichandran \cite{nr,nr1}, Partyka {\it et al.} \cite{psz},
Ponnusamy and Sairam Kaliraj \cite{pk,ps2,ps1},
Sun {\it et al.} \cite{sjr,srj}, Wang {\it et al.} \cite{wll,wsj}
derived several criteria for univalency, or quasiconformality, involving planar harmonic mappings.

Let $\mathcal{A}$ denote the class of functions $h$ of the form $$h(z)=z+\sum_{k=2}^{\infty}a_kz^k,$$
which are analytic in $\D$. Also let $\mathcal{G}(\alpha)$ be the subclass of $\mathcal{A}$ whose members satisfy the inequality
\begin{equation}\label{101}{\rm Re}\left(1+\frac{zh''(z)}{h'(z)}\right)<\alpha\quad\left(\alpha>1;\, z\in\D\right).\end{equation}
For convenience, we write $\mathcal{G}(3/2)=:\mathcal{G}$. The class $\mathcal{G}$ plays an important role in the analytic function theory.

We observe that the function class $\mathcal{G}(\alpha)$ was studied extensively by Kargar \textit{et al.} \cite{kpe}, Kanas \textit{et al.} \cite{kmp}, Maharana \textit{et al.} \cite{mps}, Obradovi\'{c} \textit{et al.} \cite{opw}, Ponnusamy and Sahoo \cite{ps} and Ponnusamy \textit{et al.} \cite{psw} for differential purposes. It is known that the functions in $\mathcal{G}(\alpha)$ are starlike in $\D$ for $\alpha\in(1,3/2]$ (see Ponnusamy and Rajasekaran \cite{pr}, Singh and Singh \cite{ss}), whereas not univalent in $\D$ for $\alpha\in(3/2,+\infty)$ (see \cite{opw}).

Recently, Mocanu \cite{m5} posed the following conjecture.

\begin{con}\label{c111} Let $$\mathcal{M}=\left\{f=h+\overline{g}\in \mathcal{H}:\ g'=zh'\ {\it and}\
{\rm Re}\left(1+\frac{zh''(z)}{h'(z)}\right)>-\frac{1}{2}\quad(z\in\D)\right\}.$$ Then $\mathcal{M}\subset\mathcal{S}_\mathcal{H}^0$.
\end{con}

By making use of
the classical results of close-to-convexity (see Kaplan \cite{k}) and harmonic close-to-convexity (see Clunie and Sheil-Small \cite{cs}), Bshouty and Lyzzaik \cite{bl} have proved Conjecture \ref{c111} by established the following stronger result.
\vskip.10in
\noindent{\bf Theorem A.}\  $\mathcal{M}\subset\mathcal{C}_\mathcal{H}^0$.
\vskip.10in

For more recent general results on the convexity, starlikeness and close-to-convexity of harmonic mappings, we refer the readers to \cite{mp,bjj,gv1,gv,kpv,koh1,koh2,lh,m5,ps1,wlrs}.
Recall the following criterion for harmonic close-to-convexity due to Abu Muhanna and Ponnusamy \cite[Corollary 3]{mp}.
\vskip.10in
\noindent{\bf Theorem B.}\  \textit{Let $h$ and $g$ be normalized analytic functions in $\D$ such that $${\rm Re}\left(1+\frac{zh''(z)}{h'(z)}\right)<\frac{3}{2},$$ and $$
g'(z)=\lambda z^nh'(z)
\quad \left(0<\abs{\lambda}\leq \frac{1}{n+1};\, n\in\N:=\{1,2,3,\ldots\}\right).$$ Then the harmonic mapping $f=h+\overline{g}$ is univalent and close-to-convex in $\D$.}
\vskip.10in

Motivated essentially by Theorem B and the definition of quasiconformal harmonic mappings, we introduce and investigate the following subclass $\mathcal{F}(\alpha,\lambda,n)$ of quasiconformal close-to-convex harmonic mappings.

\begin{dfn}
{\rm A harmonic mapping $f=h+\overline{g}\in\mathcal{H}$ is said to be in the class $\mathcal{F}(\alpha,\lambda,n)$ if
$h$ and $g$ satisfy the
conditions \begin{equation}\label{113}{\rm Re}\left(1+\frac{zh''(z)}{h'(z)}\right)<\alpha \quad\left(1<\alpha\leq\frac{3}{2}\right),\end{equation} and
\begin{equation}\label{114}
g'(z)=\lambda z^nh'(z)
\quad \left(\lambda\in\C\ {\rm with}\ \abs{\lambda}\leq \frac{1}{n+1};\, n\in\N\right).\end{equation}}
\end{dfn}

For simplicity, we denote the class $\mathcal{F}(\alpha,\lambda,1)$ by $\mathcal{F}(\alpha,\lambda)$. The image of $\D$ under the mapping $$f(z)=z-\frac{1}{2}z^2+\overline{\frac{1}{4}z^2-\frac{1}{6}z^3}\in\mathcal{F}\left({3}/{2}, {1}/{2}\right)$$ is presented as Figure \ref{fig11}.

\begin{figure}[ht]
 \centering
\includegraphics[width=4.61in]{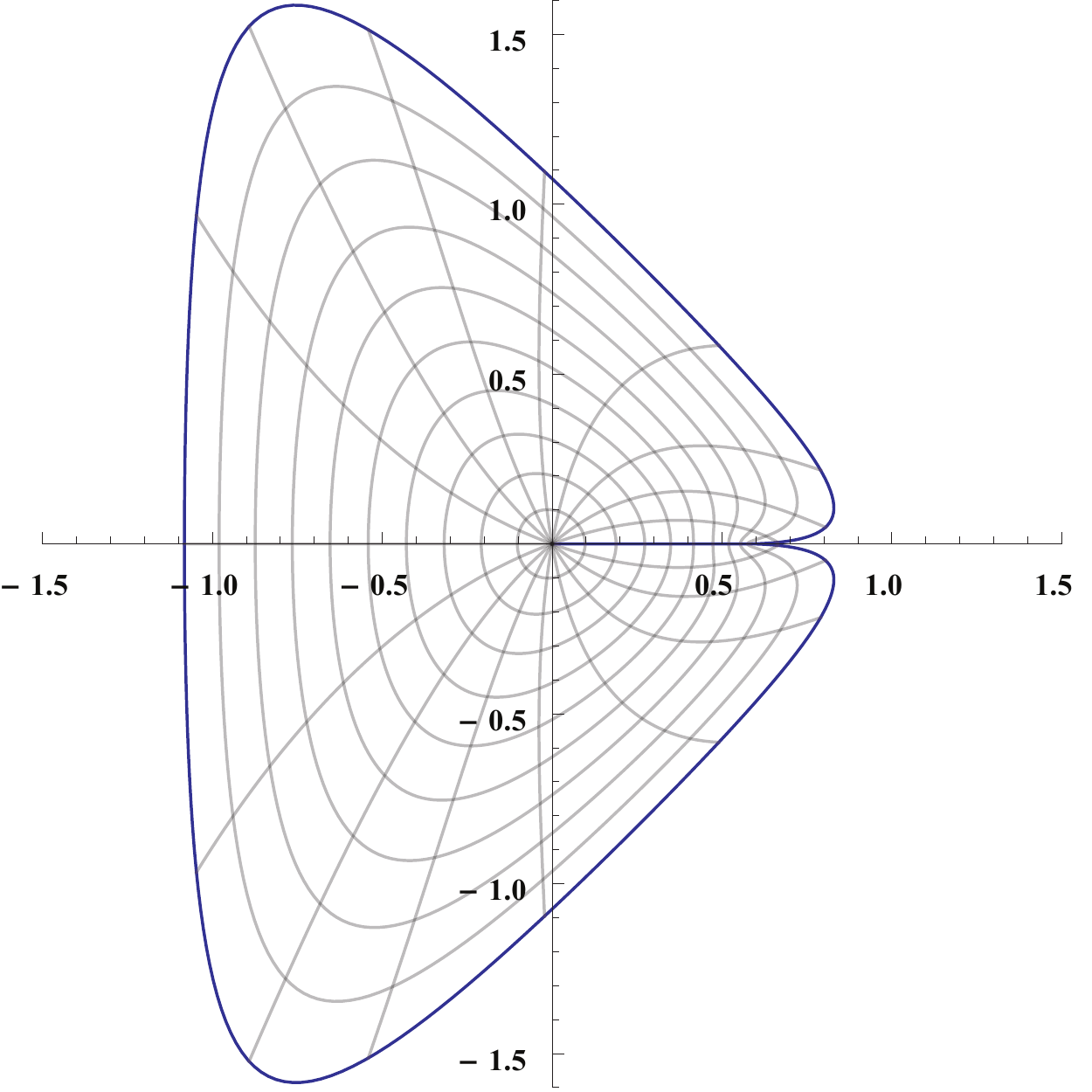}
\caption{The image of $\D$ under the mapping $f(z)=z-\frac{1}{2}z^2+\overline{\frac{1}{4}z^2-\frac{1}{6}z^3}$.}\label{fig11}
\end{figure}

This paper is organized as follows. In Section 2, we provide a counterexample to illustrate the non-univalency of the class $\mathcal{G}(\alpha)$ for $\alpha\in(3/2,2)$. In Section 3, we prove several basic properties of the class $\mathcal{F}(\alpha,\lambda,n)$ of quasiconformal close-to-convex harmonic mappings with starlike analytic part, such results as coefficient inequalities, an integral representation, a growth theorem, an area theorem, and radii of close-to-convexity of partial sums of the class, are derived.
\vskip.20in

\section{Non-univalency of the class $\mathcal{G}(\alpha)$ for $\alpha\in(3/2,+\infty)$}

Obradovi\'{c} \textit{et al.} \cite{opw} stated that the class $\mathcal{G}(\alpha)$ is not univalent in $\D$ for $\alpha\in(3/2,+\infty)$, but they did not give detailed proof about the non-univalency. We note that Kargar \textit{et al.} \cite{kpe} given a counterexample to prove the class $\mathcal{G}(\alpha)$ is not univalent in $\D$ for $\alpha\in[2,+\infty)$, in this section, we shall give a counterexample to illuminate the non-univalency of the class $\mathcal{G}(\alpha)$ for $\alpha\in(3/2,2)$.

\begin{theorem}\label{t001}
$\mathcal{G}(\alpha)\not\subset \mathcal{S}$ for $\alpha\in(3/2,+\infty)$.
\end{theorem}

\begin{proof}
We consider the analytic function $h_{\beta}\in \mathcal{A}$ given by
$$h_{\beta}(z)=\frac{1}{\beta}\left[1-(1-z)^{\beta}\right]\quad\left(2<\beta<3;\ z\in\D\right).$$ It follows that $$1+\frac{zh_{\beta}''(z)}{h_{\beta}'(z)}=\frac{1-\beta z}{1-z},$$ and therefore,  $${\rm Re}\left(1+\frac{zh_{\beta}''(z)}{h_{\beta}'(z)}\right)<\frac{1+\beta}{2}\quad\left(\frac{3}{2}<\frac{1+\beta}{2}<2\right),$$ which implies that
 $$h_{\beta}\in\mathcal{G}((1+\beta)/2)=\mathcal{G}(\alpha)\quad\left(\frac{3}{2}<\alpha<2\right).$$

In what follows, we shall prove that the function $h_{\beta}$ is not univalent in $\D$. It easily to verify that $h_{\beta}$ have real
coefficients, and thus, $h_{\beta}(z)=\overline{h_{\beta}(\overline{z})}$ for all $z\in\D$. In particular, we see that
$${\rm Re}\left(h_{\beta}\left(r e^{i \theta}\right)\right)={\rm Re}\left(h_{\beta}\left(r e^{-i \theta}\right)\right)
$$
for some $r\in (0,1)$ and $\theta\in (-\pi,0)\cup(0,\pi)$.

It is sufficient to show that there exist $r_{0}\in (0,1)$ and $\theta_{0}\in (-\pi,0)\cup(0,\pi)$ such that
$${\rm Im}\left(h_{\beta}\left(r_{0} e^{i \theta_{0}}\right)\right)={\rm Im}\left(h_{\beta}\left(r_{0} e^{-i \theta_{0}}\right)\right)=0.
$$
In view of
\begin{equation*}
{\rm Im}\left(h_{\beta}(z)\right)
={\rm Im}\left(\frac{1-(1-z)^{\beta}}{\beta}\right)
=-{\rm Im}\left(\frac{e^{\beta\log(1-z)}}{\beta}\right),
\end{equation*}
we see that
\begin{equation*}
\begin{split}
{\rm Im}\left(h_{\beta}\left(r e^{i\theta}\right)\right)&=-{\rm Im}\left(\frac{e^{\beta\log\left(1-r e^{i\theta}\right)}}{\beta}\right)\\
&=-\frac{e^{\beta\log|1-r e^{i\theta}|}}{\beta}\sin \left[\beta\arg\left(1-r e^{i\theta}\right)\right],
\end{split}
\end{equation*}
and
\begin{equation*}
\begin{split}
-{\rm Im}\left(h_{\beta}\left(r e^{-i\theta}\right)\right)=\frac{e^{\beta\log|1-r e^{-i\theta}|}}{\beta}\sin \left[\beta\arg\left(1-r e^{-i\theta}\right)\right]={\rm Im}\left(h_{\beta}\left(r e^{i\theta}\right)\right).
\end{split}
\end{equation*}
By noting that $$\arg\left(1-r e^{i\theta}\right)\in\left(-\frac{\pi}{2},0\right)\cup \left(0,\frac{\pi}{2}\right),
$$ we deduce that for each $\beta\in(2, 3)$, there exist $r_{0}\in (0,1)$ and $\theta_{0}\in (-\pi,0)\cup(0,\pi)$ such that
$$\sin\left[\beta\arg\left(1-r_{0} e^{i\theta_{0}}\right)\right]=0.
$$
It follows that $${\rm Im}\left(h_{\beta}\left(r_{0} e^{i \theta_{0}}\right)\right)={\rm Im}\left(h_{\beta}\left(r_{0} e^{-i \theta_{0}}\right)\right)=0.$$ Therefore, we see that there exist two distinct points $z_1 = r_{0} e^{i \theta_{0}}$ and $z_2 = r_{0} e^{-i \theta_{0}}$ in $\D$ such that $h_{\beta}(z_1) = h_{\beta}(z_2)$, which shows that the function $h_{\beta}(z)$ is not univalent in $\D$.
Thus, we deduce that the class $\mathcal{G}(\alpha)$ always contains a non-univalent function for each $\alpha\in(3/2,2)$.

Moreover, by noting that the class $\mathcal{G}(\alpha)$ is not univalent in $\D$ for $\alpha\in[2,+\infty)$ (see \cite[Example 2.1]{kpe}), we deduce that the assertion of Theorem \ref{t001} holds.
\end{proof}

To illustrate our counterexample, we present the image domain of $\D$ under the function $h_{5/2}(z)={2}/{5}\left[1-(1-z)^{5/2}\right]$ (see Figure \ref{fig1}).
\begin{figure}[ht]
 \centering
\includegraphics[width=4.61in]{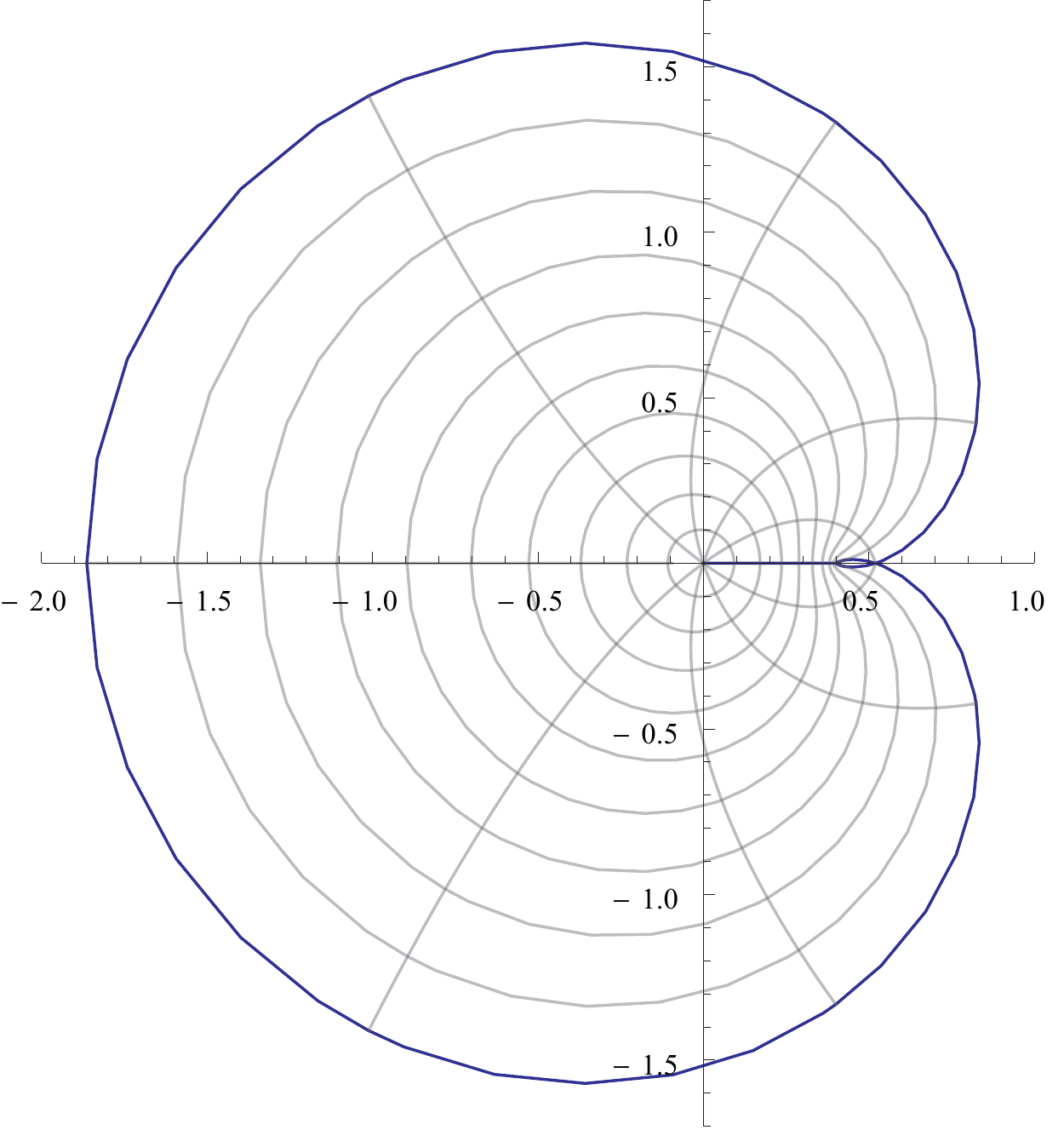}
\caption{The image of $\D$ under the function $h_{5/2}(z)={2}/{5}\left[1-(1-z)^{5/2}\right]$.}\label{fig1}
\end{figure}

\vskip.20in

\section{Properties and characteristics of the class $\mathcal{F}(\alpha,\lambda,n)$}

Let us recall the following lemma, due to Obradovi\'{c} \textit{et al.} \cite{opw}, in a slightly modified form, which will be required in the proof of Theorem \ref{t1}.

\begin{lemma}\label{lem1}
If $h(z)=z+\sum_{k=2}^{\infty}a_kz^k$ satisfies the condition \eqref{101} with $1<\alpha\leq 3/2$, then
\begin{equation}\label{41}\abs{a_k}\leq\frac{2(\alpha-1)}{(k-1)k}\quad(k\geq 2),\end{equation}
with the extremal function given by
$$h(z)=\int_0^z\left(1-t^{k-1}\right)^{\frac{2(\alpha-1)}{k-1}}dt\quad(k\geq 2).$$
\end{lemma}

\begin{theorem}\label{t1}
Let $f=h+\overline{g}\in\mathcal{F}(\alpha,\lambda,n)$ be of the form \eqref{111}. Then
the coefficients $a_k\ (k\geq 2)$ of $h$ satisfy \eqref{41},
furthermore, the coefficients $b_k\ (k=n+1, n+2, \ldots;\, n\in\N)$ of $g$ satisfy \begin{equation}\label{42}\abs{b_{n+1}}\leq\frac{\abs{\lambda}}{n+1}\quad(n\in\N)\  \ {\it and} \ \ \abs{b_{k+n}}\leq\frac{2\abs{\lambda}(\alpha-1)}{(k-1)k(k+n)}\quad\left(k\in\N\setminus\{1\};\, n\in\N\right).\end{equation}
The bounds are sharp for the extremal function given by
$$f(z)=\int_0^z\left(1-t^{k-1}\right)^{\frac{2(\alpha-1)}{k-1}}dt+\overline{\int_0^z\lambda t^n\left(1-t^{k-1}\right)^{\frac{2(\alpha-1)}{k-1}}dt}\quad\left(n\in\N\right).$$
\end{theorem}

\begin{proof}
Comparing the coefficients of $z^{k+n-1}$ of both sides in \eqref{114}, we obtain \begin{equation}\label{6}(k+n)b_{k+n}=\lambda k a_k\quad (k,\,n\in\N;\, a_1=1).\end{equation} Combining Lemma \ref{lem1} with \eqref{6}, we readily get the desired coefficient inequalities \eqref{42} of Theorem \ref{t1}.
\end{proof}





The Fekete-Szeg\"{o} functional for
$\abs{a_{3}-\delta a_{2}^{2}}$ of the class $\mathcal{G}(\alpha)$ with $\alpha\in(1,3/2]$ was discussed by Obradovi\'{c} \textit{et al.} \cite{opw}, which will be useful in the proof of the upper bounds for $\abs{b_{3}-\delta b_{2}^{2}}$
of functions in the class $\mathcal{F}(\alpha,\lambda)$. We here present its modified form.

\begin{lemma}\label{lem2} Let $f\in\mathcal{G}(\alpha)$ with $\alpha\in(1,3/2]$. Then
\begin{equation}\label{31}
\abs{a_{3}-\delta a_{2}^{2}}\leq\left\{\begin{array}{cc}
\frac{\alpha-1}{3}\abs{3+\delta-(2+\delta)\alpha}
&\left(\abs{\delta-\frac{3-2\alpha}{3(\alpha-1)}}\geq\frac{1}{3(\alpha-1)}\right), \\\\
\frac{\alpha-1}{3}\ \  &\left(\abs{\delta-\frac{3-2\alpha}{3(\alpha-1)}}<\frac{1}{3(\alpha-1)}\right).\end{array}\right.
\end{equation}
Equality in the Fekete-Szeg\"{o} functional is attained in each case.
\end{lemma}

\begin{theorem}\label{t2}
Let $f\in\mathcal{F}(\alpha,\lambda)$ be of the form \eqref{111}. Then
\begin{equation}\label{32}
\big|b_{3}-\delta b_{2}^{2}\big|\leq \frac{2(\alpha-1)|\lambda|}{3}+\frac{|\delta||\lambda|^{2}}{4}.
\end{equation}
The inequality is sharp.
\end{theorem}

\begin{proof}
By noting that $g'(z)=\lambda zh'(z)$ for $f\in\mathcal{F}(\alpha,\lambda)$, we have
\begin{equation*}
\sum_{k=2}^{\infty}kb_{k}z^{k-1}=\lambda\sum_{k=1}^{\infty}ka_{k}z^{k}\quad(a_1=1).
\end{equation*}
Clearly, we see that \begin{equation}\label{302}b_{2}=\frac{1}{2}\lambda a_{1}=\frac{1}{2}\lambda\ \ \textrm{and}\ \
b_{3}=\frac{2}{3}\lambda a_{2}.\end{equation} Therefore, by virtue of \eqref{31} and \eqref{302}, we obtain
\begin{equation*}
\big|b_{3}-\delta b_{2}^{2}\big|
=\left|\frac{2}{3}\lambda a_{2}-\frac{1}{4}\delta\lambda^{2}\right|
\leq \frac{2|\lambda||a_{2}|}{3}+\frac{|\delta||\lambda|^{2}}{4}
\leq \frac{2(\alpha-1)|\lambda|}{3}+\frac{|\delta||\lambda|^{2}}{4}.
\end{equation*}The proof of Theorem \ref{t2} is thus completed.
\end{proof}

By setting $\delta=1$ in Lemma \ref{lem2}, respectively Theorem \ref{t2}, we get the
Zalcman type coefficient inequalities of the class $\mathcal{F}(\alpha,\lambda)$ for the case $k=2$. For recent developments on this topic (see Li and Ponnusamy \cite{lp} and the references therein).

\begin{cor}
Let $f\in\mathcal{F}(\alpha,\lambda)$ be of the form \eqref{111}. Then
$$
\abs{a_{3}-a_{2}^{2}}\leq\left\{\begin{array}{cc}
\frac{\alpha-1}{3}\abs{4-3\alpha}
&\left(\abs{1-\frac{3-2\alpha}{3(\alpha-1)}}\geq\frac{1}{3(\alpha-1)}\right), \\\\
\frac{\alpha-1}{3}\ \  &\left(\abs{1-\frac{3-2\alpha}{3(\alpha-1)}}<\frac{1}{3(\alpha-1)}\right),\end{array}\right.$$
and
$$\big|b_{3}-b_{2}^{2}\big|\leq \frac{2(\alpha-1)|\lambda|}{3}+\frac{|\lambda|^{2}}{4}\leq\frac{11}{48}.
$$
The inequalities are sharp.
\end{cor}

Now, we give an integral representation of the mapping $f\in\mathcal{F}(\alpha,\lambda,n)$.

\begin{theorem}\label{t3}
Let $f\in\mathcal{F}(\alpha,\lambda,n)$.
Then
\begin{equation*}\begin{split}f(z)=&\int_0^z\exp\left(2(1-\alpha)\int_0^{\zeta}
\frac{\varpi(t)}{t(1-\varpi(t))}dt\right)d\zeta\\&\quad\quad+\overline{\lambda\int_0^z{{\zeta}^{n}}\cdot\exp\left(2(1-\alpha)\int_0^{\zeta}
\frac{\varpi(t)}{t(1-\varpi(t))}dt\right)d\zeta},\end{split}
\end{equation*}where $\varpi$ is the Schwarz function with
$\varpi(0)=0$ and $\abs{\varpi(z)}<1\ (z\in\D)$.
\end{theorem}

\begin{proof}
Suppose that $f\in\mathcal{F}(\alpha,\lambda,n)$. It follows from \eqref{113} that
\begin{equation}\label{37}1+\frac{zh''(z)}{h'(z)}\prec\frac{1-(2\alpha-1)z}{1-z}\quad(z\in\D),
\end{equation}
where $``\prec"$ denotes the familiar subordination of analytic functions.
By virtue of \eqref{37}, we see that
\begin{equation}\label{38}1+\frac{zh''(z)}{h'(z)}=\frac{1-(2\alpha-1)\varpi(z)}{1-\varpi(z)}\quad(z\in\D),
\end{equation}
where $\varpi$ is the Schwarz function with
$\varpi(0)=0$ and $\abs{\varpi(z)}<1\ (z\in\D)$.
From \eqref{38}, we have
$$\frac{(zh'(z))'}{zh'(z)}-\frac{1}{z}=\frac{2(1-\alpha)\varpi(z)}{z(1-\varpi(z))},$$
which, upon integration, yields
\begin{equation}\label{39}\log(h'(z))=2(1-\alpha)\int_0^z\frac{\varpi(t)}{t(1-\varpi(t))}dt.\end{equation}
We thus find from \eqref{39} that
\begin{equation}\label{310}h(z)=\int_0^z\exp\left(2(1-\alpha)\int_0^{\zeta}
\frac{\varpi(t)}{t(1-\varpi(t))}dt\right)d\zeta.\end{equation}
Combining \eqref{114} with \eqref{310}, we obtain
\begin{equation}\label{311}g(z)=\lambda\int_0^z{{\zeta}^{n}}\cdot\exp\left(2(1-\alpha)\int_0^{\zeta}
\frac{\varpi(t)}{t(1-\varpi(t))}dt\right)d\zeta.\end{equation}
Thus, the assertion of Theorem \ref{t3} follows from \eqref{310} and \eqref{311}.
\end{proof}

\begin{rem}
{\rm Theorem \ref{t3} provides a direct integration method for constructing quasiconformal close-to-convex harmonic mappings by choosing suitable Schwarz functions $\varpi$.}
\end{rem}

The following lemma due to Maharana \textit{et al.} \cite{mps} will play a crucial role in the proof of our last three results.

\begin{lemma}\label{lem3}
If $h\in\mathcal{G}$, then for $\abs z=r<1$, the following statements are true.
\begin{enumerate}
\item
$$\abs{\frac{zh''(z)}{h'(z)}}\leq\frac{r}{1-r}.$$
The inequality is sharp and equality is attained for the function \begin{equation}\label{4001}h(z)=z-\frac{z^2}{2}.\end{equation}
\item
\begin{equation}\label{4002}1-r\leq\abs{h'(z)}\leq 1+r.\end{equation} The inequalities are sharp and equalities are attained for the function given by \eqref{4001}.
\item
If $h(z)=\mathcal{S}_n(z)+\Sigma_n(z)$, with $\Sigma_n(z)=\sum_{k=n+1}^{\infty}a_kz^k$, then
$$\abs{\Sigma'_n(z)}\leq r^n\phi(r,1,n)\ {\it and}\ \abs{z\Sigma''_n(z)}\leq\frac{r^n}{1-r},$$
where $\phi(r,1,n)$ is the unified zeta function which is defined by the series
$$\phi(z,s,a)=\sum_{k=0}^{\infty}\frac{z^k}{(k+a)^s}\quad(\abs z<1;\, \Re(s)>1;\, a\neq 0,-1,-2,\ldots).$$
\end{enumerate}
\end{lemma}

We now give the growth theorem for the class $\mathcal{F}(\alpha,\lambda,n)$.

\begin{theorem}\label{t401}
Let $f\in\mathcal{F}(\alpha,\lambda,n)$.
Then
\begin{equation}\label{411}
\begin{split}
&r \left[|\lambda|  \left(\frac{r}{n+2}-\frac{1}{n+1}\right) r^n-\frac{r}{2}+1\right]
\leq \abs{f(z)}\\
&\qquad\qquad\qquad\leq r \left[|\lambda|  \left(\frac{r}{n+2}+\frac{1}{n+1}\right) r^n+\frac{r}{2}+1\right].
\end{split}
\end{equation}
The inequalities are sharp.
\end{theorem}

\begin{proof}
Assume that $f=h+\overline{g}\in\mathcal{F}(\alpha,\lambda,n)$. By observing that $h\in\mathcal{G}$,
we know that \eqref{4002} holds.
Also, let $\Gamma$ be the line segment joining $0$ and $z$, then
\begin{equation}\label{413}
\begin{split}
\abs{f(z)}&=\abs{\int_{\Gamma}\frac{\partial f}{\partial \xi}d\xi
+\frac{\partial f}{\partial \overline{\xi}}d\overline{\xi}}\\
&\leq \int_{\Gamma}\left(\abs{h'(\xi)}+\abs{g'(\xi)}\right)\abs{d\xi}\\
&= \int_{\Gamma}\left(1+\abs{\lambda}\abs{\xi}^n\right)\abs{h'(\xi)}\abs{d\xi}\\
&\leq \int_{0}^{r}{(1+\xi)(1+|\lambda|\xi^n)}d\xi\\
&=\frac{1}{2} r \left[2 |\lambda|  \left(\frac{r}{n+2}+\frac{1}{n+1}\right) r^n+r+2\right].
\end{split}
\end{equation}
Moreover, let $\widetilde{\Gamma}$ be the preimage under $f$ of the line segment joining $0$ and $f(z)$, then we obtain
\begin{equation}\label{414}
\begin{split}
\abs{f(z)}&=\int_{\widetilde{\Gamma}}\abs{\frac{\partial f}{\partial \xi}d\xi
+\frac{\partial f}{\partial \overline{\xi}}d\overline{\xi}}\\
&\geq \int_{\widetilde{\Gamma}}\left(\abs{h'(\xi)}-\abs{g'(\xi)}\right)\abs{d\xi}\\
&= \int_{\widetilde{\Gamma}}\left(1-|\lambda||\xi|^n\right)\abs{h'(\xi)}\abs{d\xi}\\
&\geq \int_{0}^{r}{(1-\xi)(1-|\lambda|\xi^n)}d\xi\\
&=\frac{1}{2} r \left[2 |\lambda|  \left(\frac{r}{n+2}-\frac{1}{n+1}\right) r^n-r+2\right].
\end{split}
\end{equation}
It follows from \eqref{413} and \eqref{414} that
 the assertion \eqref{411} of Theorem \ref{t401} holds.
\end{proof}

Denote by
$\mathcal{A}\left(f(\D_{r})\right)$ the area of $f(\D_{r})$,
where $\D_{r}:=r\D$ for $0<r<1$. We now consider the area theorem of mappings $f$ belong to the
class $\mathcal{F}(\alpha,\lambda,n)$.

\begin{theorem}\label{t43}
Let $f\in\mathcal{F}(\alpha,\lambda,n)$.
Then for $0<r<1$, we have
\begin{equation}\label{421}
2\pi\int_{0}^{r}{\left(1-|\lambda|^{2}\xi^{2n}\right)(1-\xi)^2\xi}d\xi
\leq \mathcal{A}\left(f(\D_{r})\right)\leq
2\pi\int_{0}^{r}{\left(1-|\lambda|^{2}\xi^{2n}\right)(1+\xi)^2\xi}d\xi.
\end{equation}
\end{theorem}

\begin{proof}
Suppose that $f=h+\overline{g}\in\mathcal{F}(\alpha,\lambda,n)$. Then for $0<r<1$, we get
\begin{equation}\label{4012}
\mathcal{A}\left(f(\D_{r})\right)=\iint_{\D_{r}}\left(|h'(z)|^{2}-|g'(z)|^{2}\right)\,dx\,dy
=\iint_{\D_{r}}\left(1-|\lambda|^{2}|z|^{2n}\right)|h'(z)|^{2}\,dx\,dy.
\end{equation}
In view of \eqref{4002} and \eqref{4012}, we obtain the result of Theorem \ref{t43}.
\end{proof}


Finally, we shall discuss the radius problems of mappings $f\in\mathcal{F}(\alpha,\lambda)$.
The largest value of $r$ so that the partial sums of $f\in\mathcal{F}(\alpha,\lambda)$
are close-to-convex in $|z|<r$ are considered. For recent results on partial sums of univalent harmonic mappings (see, e.g., Chen \textit{et al.} \cite{crw}, Ghosh and Vasudevarao \cite{gv},  Li and Ponnusamy \cite{lp1,lp2,lp3}, Ponnusamy \textit{et al.} \cite{pss2}, Sun \textit{et al.} \cite{sjr}).

\begin{theorem}\label{t5}
Let $f\in\mathcal{F}(\alpha,\lambda)$ be of the form \eqref{111}. Then for each $m\geq 1,\ l\geq 2$,
$$\mathcal{S}_{m,l}(f)(z)=\sum_{k=1}^{m}a_{k}z^{k}+\overline{\sum_{k=2}^{l} b_{k}z^{k}}\quad(a_1=1)$$ is close-to-convex in $|z|<r_{c}\approx 0.503$, where $r_{c}$
is the least positive real root in the interval $(0, 1)$ of the equation:
\begin{equation}\label{50}
2+2\ln(1-r)+r\ln(1-r)-r+r^2=0.
\end{equation}
The bound  $r_{c}$ is sharp.
\end{theorem}

\begin{proof}
Let $f=h+\overline{g}\in\mathcal{F}(\alpha,\lambda)$ and $\phi=h+\varepsilon\overline{g}$ with $|\varepsilon|=1$.
We observe that
${\rm Re}\left(\varphi'(z)\right)>0$ for $\varphi\in\mathcal{A}$ implies that $\varphi$ is a close-to-convex analytic function.
Therefore, it is sufficient to show that each partial sums
\begin{equation*}
\mathcal{S}_{m,l}(\phi)(z)=\sum_{k=1}^{m}a_{k}z^{k}+\varepsilon\overline{\sum_{k=2}^{l} b_{k}z^{k}}
\end{equation*}
satisfies the condition
\begin{equation*}
{\rm Re}\left(\Gamma'_{m,l}(\phi)(z)\right)>0
\end{equation*}
in the disk $|z|<r_{c}$ for all $|\varepsilon|=1$ and $m\geq 1,\ l\geq 2$, where
\begin{equation*}
\Gamma_{m,l}(\phi)(z)=\sum_{k=1}^{m}a_{k}z^{k}+\varepsilon\sum_{k=2}^{l} b_{k}z^{k}.
\end{equation*}

In order to prove the radii of close-to-convexity for the partial sums $\mathcal{S}_{m,l}(f)(z)$, we split it into four cases to prove.
\begin{enumerate}
\item

For $m=1,2$, $l=2$, we have

$$
\Gamma_{1,2}(\phi)(z)=z+\varepsilon b_{2}z^{2},$$
 and
$$\Gamma_{2,2}(\phi)(z)=z+a_{2}z^{2}+\varepsilon b_{2}z^{2},
$$
it follows that
$$
\Gamma'_{1,2}(\phi)(z)=1+\varepsilon\lambda z,
$$ and $$
\Gamma'_{2,2}(\phi)(z)=1+2a_{2}z+\varepsilon\lambda z.
$$
Clearly, ${\rm Re}\big(\Gamma'_{1,2}(\phi)(z)\big)>0$ in $|z|<r_{1}={2}/{3}$.
By Lemma \ref{lem1}, we know that $|a_{2}|\leq \alpha-1$, thus,
\begin{equation*}\begin{split}
{\rm Re}\left(\Gamma'_{2,2}(\phi)(z)\right)&\geq 1-2|a_{2}||z|-|\lambda||z|
\\&\geq 1-[2(\alpha-1)+|\lambda|]|z|\\&\geq 1-\frac{3}{2}|z|>0  \quad (|z|<r_{1}).\end{split}
\end{equation*}

\item

For $m,l\geq 3$, we find from \eqref{113} and \eqref{114} that
\begin{equation}\label{51}
\begin{split}
&{\rm Re}\left(\Gamma'_{m,l}(\phi)(z)\right)
\\&\quad={\rm Re}\left(\mathcal{S}'_{m}(h)(z)+\varepsilon\lambda z\mathcal{S}'_{l-1}(h)(z)\right)\\
&\quad={\rm Re}\left(\left(h'(z)-\Sigma'_{m}(h)(z)\right)+\varepsilon\lambda z\left(h'(z)-\Sigma'_{l-1}(h)(z)\right)\right)\\
&\quad\geq {\rm Re}\left(h'(z)\right)-|\Sigma'_{m}(h)(z)|-|\lambda||z||h'(z)|-|\lambda||z||\Sigma'_{l-1}(h)(z)|\\
&\quad\geq {\rm Re}\left(h'(z)\right)-|\Sigma'_{m}(h)(z)|-\frac{1}{2}|z||h'(z)|-\frac{1}{2}|z||\Sigma'_{l-1}(h)(z)|.
\end{split}
\end{equation}
In view of \eqref{4002},
we obtain
\begin{equation}\label{52}
\min_{|z|=r<1}\{{\rm Re}\left(h'(z)\right)\}
\geq \min_{|z|=r<1}\{{\rm Re}\left(1-z\right)\}
\geq 1-r.
\end{equation}
From Lemma \ref{lem3}(3), for $|z|=r<1$, we know that
\begin{equation*}
\abs{\Sigma'_n(z)}
\leq\sum_{k=0}^{\infty}\frac{r^{k+n}}{k+n}=-\ln(1-r)-\sum_{k=1}^{n-1}\frac{r^k}{k}=:\Delta(n),
\end{equation*}
and
\begin{equation*}
\Delta(n+1)-\Delta(n)=-\frac{r^{n}}{n}<0\quad(n\geq 2).
\end{equation*}
Therefore, $\Delta(n)$ is a decreasing function of $n$. For all $m,l\geq 3$, we see that
\begin{equation}\label{53}
\Delta(m)\leq \Delta(3)=-\ln(1-r)-r-\frac{r^2}{2},
\end{equation} and \begin{equation}\label{54}
\Delta(l-1)\leq \Delta(2)=-\ln(1-r)-r.
\end{equation}
Moreover, it follows from Lemma \ref{lem3}(2) that
\begin{equation}\label{55}
|z||h'(z)|\leq |z|(1+|z|)=r(1+r)\quad(|z|=r<1).
\end{equation}
From the relationships \eqref{51}, \eqref{52}, \eqref{53}, \eqref{54} and \eqref{55}, it follows that
\begin{equation*}
\begin{split}
{\rm Re}\left(\Gamma'_{m,l}(\phi)(z)\right)
\geq 1+\ln(1-r)-\frac{r}{2}+\frac{1}{2}r\ln(1-r)+\frac{1}{2}r^2>0
\end{split}
\end{equation*}
for all $m,l\geq 3$ and $|z|=r<r_{2}\approx 0.503$, where $r_{2}$
is the least positive root in the interval $(0, 1)$ of the equation:
\begin{equation*}
2+2\ln(1-r)+r\ln(1-r)-r+r^2=0.
\end{equation*}

\item
For $m=1,2$, $l\geq 3$, we see that
\begin{equation*}
\begin{split}
{\rm Re}\left(\Gamma'_{2,l}(\phi)(z)\right)
&={\rm Re}\left(\mathcal{S}'_{2}(h)(z)+\varepsilon \mathcal{S}'_{l}(g)(z)\right)\\
&={\rm Re}\left(1+2a_{2}z+\varepsilon\lambda z \mathcal{S}'_{l-1}(h)(z)\right)\\
&\geq 1-2|a_{2}||z|-|\lambda||z||h'(z)|-|\lambda||z||\Sigma'_{l-1}(h)(z)|\\
&\geq 1-\frac{1}{2}|z|-\frac{1}{2}|z||h'(z)|-\frac{1}{2}|z||\Sigma'_{l-1}(h)(z)|.
\end{split}
\end{equation*}
From \eqref{53} and \eqref{54}, we know that
\begin{equation*}
{\rm Re}\left(\Gamma'_{2,l}(\phi)(z)\right)
\geq 1-\frac{1}{2}r-\frac{r(1+r)}{2}+\frac{1}{2}r[\ln(1-r)+r]>0
\end{equation*}
for all $l\geq 3$ and $|z|=r<r_{3}\approx 0.653575$, where $r_{3}$
is the least positive root in the interval $(0, 1)$ of the equation:
\begin{equation*}
2-2r+r\ln(1-r)=0.
\end{equation*}
Similarly, for all $l\geq 3$ and $|z|=r<r_{3}$, we have
\begin{equation*}
\begin{split}
{\rm Re}\left(\Gamma'_{1,l}(\phi)(z)\right)
\geq 1-\frac{1}{2}|z||h'(z)|-\frac{1}{2}|z||\Sigma'_{l-1}(h)(z)|\geq
1-\frac{r}{2}+\frac{r}{2}\ln(1-r)
>0.
\end{split}
\end{equation*}

\item
For $m\geq 3, l=2$, we deduce from \eqref{52} and \eqref{53} that
\begin{equation*}
\begin{split}
{\rm Re}\left(\Gamma'_{m,2}(\phi)(z)\right)
&={\rm Re}\left(\mathcal{S}'_{m}(h)(z)+\varepsilon \mathcal{S}'_{2}(g)(z)\right)\\
&\geq {\rm Re}\left(h'(z)\right)-|\Sigma'_{m}(h)(z)|-|\lambda||z|\\
&\geq {\rm Re}\left(h'(z)\right)-|\Sigma'_{m}(h)(z)|-\frac{1}{2}|z|\\
&\geq
1-\frac{1}{2}r+\ln(1-r)+\frac{r^2}{2}
>0,
\end{split}
\end{equation*}
where $|z|=r<r_{4}\approx 0.584628$, where $r_{4}$ is the least positive root in the interval $(0, 1)$ of the equation:
\begin{equation*}
2-r+2\ln(1-r)+{r^2}=0.
\end{equation*}
\end{enumerate}

By setting
\begin{equation*}
r_{c}:=\min\{r_{1},r_{2},r_{3},r_{4}\}=r_{2},
\end{equation*}
we see that ${\rm Re}\left(\Gamma'_{m,l}(\phi)(z)\right)>0$
for all $|z|<r_{c}$ and $m\geq 1,\ l\geq 2$.
The proof of Theorem \ref{t5} is thus completed.
\end{proof}

\vskip .20in
\begin{center}{\sc Acknowledgments}
\end{center}

\vskip .05in
The present investigation was supported by the \textit{Key Project of Education Department of Hunan Province} under Grant no.
19A097, and
the \textit{National
Natural Science Foundation} under Grant no. 11961013 of the P. R. China.

\end{document}